\documentclass[a4paper, 11pt, reqno]{amsart}
\usepackage[english]{babel}
\usepackage[utf8]{inputenc}
\usepackage[T1]{fontenc}
\usepackage{amsthm, amsmath, amssymb, enumerate, enumitem, mathtools, lmodern, microtype, mathrsfs, listings, tikz-cd, comment, float}
\usepackage{a4wide}
\usepackage{hyperref}
\usepackage{nicefrac}

\usepackage{bbm}

\allowdisplaybreaks 
\let\epsilon\varepsilon

\theoremstyle{plain}
\newtheorem{thm}{Theorem}[section]
\newtheorem{prop}[thm]{Proposition}
\newtheorem{lemma}[thm]{Lemma}

\newtheorem{conj}[thm]{Conjecture}
\theoremstyle{definition}

\theoremstyle{remark}
\newtheorem*{rmk}{Remark}
\newtheorem*{rmks}{Remarks}

\binoppenalty=\maxdimen
\relpenalty=\maxdimen

\numberwithin{equation}{section}
\setlist{nosep}
\setlist{noitemsep}

\newcommand{\C}{\mathbb{C}}

\newcommand{\N}{\mathbb{N}}

\DeclareMathOperator{\lcm}{lcm}

\newcommand{\up}{\mathrm{up}}

\makeatletter
\def\paragraph{\@startsection{paragraph}{4}%
  \z@\z@{-\fontdimen2\font}%
  {\normalfont\bfseries}}
\makeatother

\usepackage{enumitem}
\setlist[enumerate]{label=\rm\roman*),itemindent=0pt,leftmargin=1cm,itemsep=0.5ex}

\title[D'Arcais is not Hurwitz]{A Positive Proportion of the Reduced \\ D'Arcais Polynomials is not Hurwitz}

\author[Charlton]{Steven Charlton}
\address{}
\email{mail@stevencharlton.net}

\author[Heim]{Bernhard Heim}
\address{Department of Mathematics and Computer Science, Division of Mathematics, University of Cologne,
	Weyertal 86-90, 50931 Cologne, Germany}
\email{bheim@uni-koeln.de}
\email{jstumpen@math.uni-koeln.de}

\author[Neuhauser]{Markus Neuhauser}
\address{Kutaisi International University, 5/7, Youth Avenue, Kutaisi, 4600 Georgia \& \newline\indent
Lehrstuhl für Geometrie und Analysis, RWTH Aachen University, 52056 Aachen,
Germany}
\email{markus.neuhauser@kiu.edu.ge}

\author[Stumpenhusen]{\\ Johann Stumpenhusen}

\author[Tröger]{Robert Tröger}

\dedicatory{In memoriam Florian Luca}

\date{August 19, 2026}

\begin{document}

\begin{abstract}
    Heretofore, the second and third author conjectured that the D'Arcais polynomials, related to the coefficients of the powers of the Dedekind $\eta$-function, are Hurwitz polynomials except for a root at the origin. We show that this does in fact not hold for a positive proportion of all natural numbers. 
\end{abstract}

\maketitle

\section{Introduction and statement of results}

\subsection{Coefficients of powers of the Dedekind \texorpdfstring{$\eta$}{eta}-function} The \emph{Dedekind $\eta$-function}
\begin{equation*}
    \eta(\tau) \coloneqq q^{\frac{1}{24}} \prod_{n = 1}^\infty (1 - q^n), \qquad q \coloneqq e^{2\pi i \tau},    
\end{equation*}
represents the standard example of a weight $\frac{1}{2}$ modular form whose Fourier coefficients are tightly connected to the partition function $p(n)$. Since its 24\textsuperscript{th} power, the \emph{modular discriminant} $\Delta = \sum_{n = 1}^\infty \tau(n)q^n$, generates the space of weight $12$ cusp forms, the study of the coefficients of its powers and especially their vanishing behaviour has been enjoying widespread attention. In fact, Serre \cite[Théorème 1]{Serre} showed that $\eta(q)^r$ is lacunary for even $r$ if and only if $r \in \{2,4,6,8,10,14,26\}$ and Lehmer \cite{Leh47} famously conjectured that $\tau(n) \neq 0$ for all $n$.

\subsection{D'Arcais polynomials} Fortunately, these coefficients possess a simple but structure-heavy description. The \emph{D'Arcais polynomials}  \cite{DAr13} (or Nekrasov--Okounkov polynomials \cite{Ha10,NO06,Zh22} in combinatorics) are defined via the infinite product
\begin{equation}\label{eqn:darcais:exp}
	\sum_{n=0}^\infty P_n^\sigma(X) q^n \coloneqq \prod_{m=1}^\infty (1 - q^m)^{-X} = \exp\Bigg( X \sum_{j=1}^\infty \sigma(j) \frac{q^j}{j} \Bigg) \,,
\end{equation}
where \( \sigma(n) = \sigma_1(n) \), with \( \sigma_a(n) \coloneqq \sum_{d \mid n} d^a \) the generalised sum-of-divisors function. 
While Lehmer's conjecture remains unsolved, the focus shifted to other analytic properties as log-concavity or unimodality, as done by Abdesselam \cite{Abdessel23}, Abdesselam et al. \cite{AbPruDoVe}, Starr \cite{Starr}, and some of the authors \cite{CharltonHeimStum,HeimNeu19Hooklength,HeimNeu20,HeimNeu25,HeimStum26,stumpenhusen2026logconcavitydarcaispolynomialsnormalised}. Let
\[
f(X)=\sum_{m=0}^{d} a_mX^m,
\qquad a_m>0,
\]
be a polynomial of degree $d$ with positive real coefficients. The polynomial $f$ is called a
\emph{Hurwitz polynomial} if all its zeros lie in the open left complex half-plane. This kind of polynomials appears as characteristic polynomial in stability theory \cite[Chapter I]{Routh}. The second and third author conjectured the following concerning $R_n^\sigma(X) := \frac{P_n^\sigma(X)}{X}$.

\begin{conj}[\protect{\cite[Conjecture 2 (ii)]{HeimNeu19Hooklength}}]\label{conj:HeimNeuhauserHurwitz}
    For all $n \in \N$, the polynomial $R_n^\sigma(X)$ is Hurwitz.
\end{conj}

This conjecture was motivated by the first $1\, 000$ D'Arcais polynomials being Hurwitz, see Figure \ref{fig:ComplexRoots} for the non-real roots in those cases.  Since the coefficients of \( R_n^\sigma(X) \) are strictly positive, any real root must be on the negative real axis; we eliminate these from the following figure for the sake of clarity.

\begin{figure}[H]
    \centering
    \includegraphics[width=0.9\textwidth]{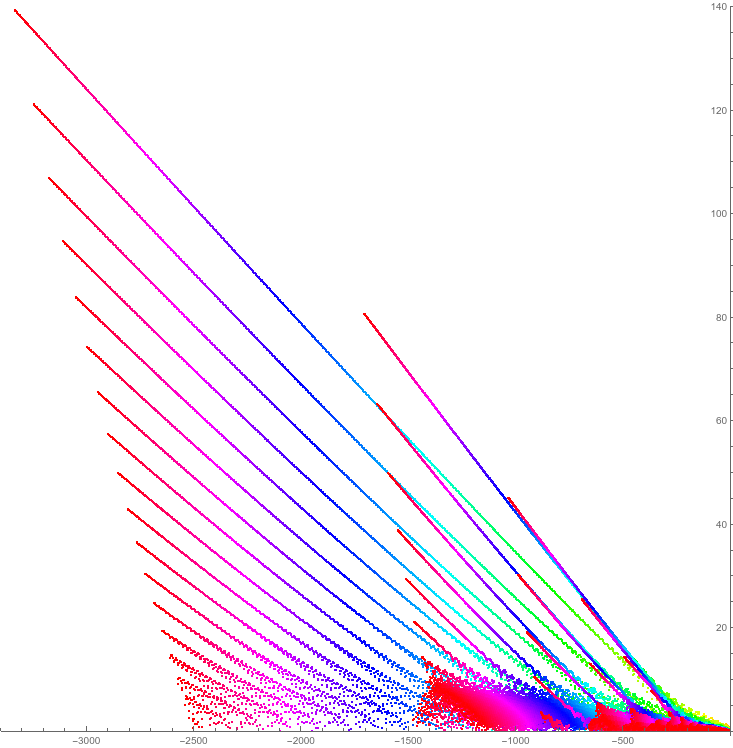}
    \caption{Distribution of non-real roots in the upper half-plane for $P_n^\sigma(X)$ with $n \leq 1 \, 000 \,$. The colour of the points runs from orange for $n = 1$, through yellow, green, cyan, blue, and purple to red for $n = 1 \, 000 \,$, as follows.}
    \includegraphics[width=0.7\textwidth, clip, trim=0 0 0 330]{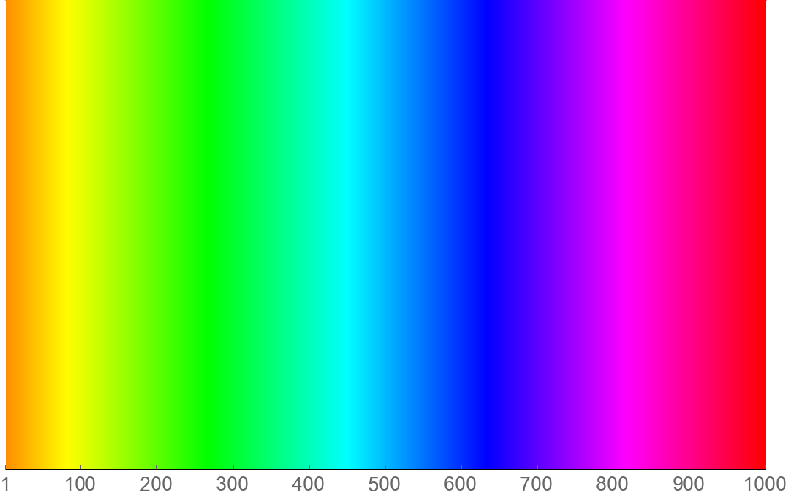}
    \label{fig:ComplexRoots}
\end{figure}

\subsection{Results} Using estimates from previous work by the first, second, and fourth author, we offer the following negative answer to Conjecture \ref{conj:HeimNeuhauserHurwitz}.

\begin{thm}\label{thm:DArcaisNotHurwitz}
    Let $\mathcal{A}(n) := \#\left\{ k \leq n : R_k^\sigma(X) \text{ is not Hurwitz.}\right\}$. Then
    \begin{equation*}
        \liminf_{n \to \infty} \frac{\mathcal{A}(n)}{n} > 0\, .
    \end{equation*}
\end{thm}

\begin{rmk}
  In Section \ref{sec:proofhurwitz}, we establish a lower bound on this limit inferior, by showing that a counterexample to \( P_k^\sigma(X) \) being Hurwitz must occur by the point \( \ell! \), where 
\begin{align*}
\ell = \smash[b]{\left\lceil\exp\left(\frac{25\,046\,441}{116\, 352}\right)\right\rceil} = 
3\,&077\,030\,883\,517\,849\,244\,056\,113\,561\,601\,934\,895\,798\,281\, 720\,\\[-1ex]&971\,738\,852\,399\,001\,591\,620\,287\,219\,488\,043\,505\,919\,072\,573\,770\, .
\end{align*}

\end{rmk}

\section{Preliminaries}

We collect some preliminary properties that will be used while proving our main result.

\subsection{The Hurwitz--Routh criterion} For $1\leq k\leq d$, let
\[
\Delta_k(f)
:=
\det\bigl(a_{d-2j+i}\bigr)_{1\leq i,j\leq k},
\]
where $a_m:=0$ for $m<0$ and for $m>d$. Thus,
\[
\Delta_k(f)
=
\det
\begin{pmatrix}
 a_{d-1} & a_{d-3} & a_{d -5} & \cdots \\
 a_{d}   & a_{d-2} & a_{d-4} & \cdots \\
 0         & a_{d-1} & a_{d-3} & \cdots \\
 0         & a_{d}   & a_{d-2} & \cdots \\
 \vdots    & \vdots    & \vdots    & \ddots
\end{pmatrix}_{k \times k}.
\]
\begin{thm}[\protect{Hurwitz \cite[p. 274]{Hurwitz}}]\label{thm:Hurwitz}
    A polynomial $f$ is Hurwitz if and only if
    \begin{equation*}
        \Delta_k(f) > 0
    \end{equation*}
    for all $1 \leq k \leq d$.
\end{thm}
In particular,
\[
\Delta_1(f)=a_{d-1} > 0
\]
and
\[
\Delta_2(f)
=
\begin{vmatrix}
 a_{d-1} & a_{d-3}\\
 a_{d}   & a_{d-2}
\end{vmatrix}
=
a_{d-1}a_{d-2}-a_{d}a_{d-3} > 0  \,,
\]
are necessary conditions.  It is well known that if some minor \( \Delta_k(f) < 0 \) is strictly negative, then there is actually a root in the open right half-plane \cite{Adm2018}.

\subsection{Lower bounds for \texorpdfstring{$p_n^\sigma(1)$}{p\_n\textasciicircum{}sigma(1)} and \texorpdfstring{$p_n^\sigma(4)$}{p\_n\textasciicircum{}sigma(4)}}

Let
\[
P_n^\sigma(X) := \sum_{k = 0}^n p_n^\sigma(k)X^k.
\]
A simple elegant choice of variables yields a lower bound for the first coefficient. Define
\[
\rho(n) := \lcm\{k: 1 \leq k \leq n\}.
\]

\begin{lemma}\label{lem:BoundP(1)}
Let $n \in \N$ and $H_\ell$ be the $\ell$-th harmonic number. If $\rho(\ell)$ divides $n$, then
\[
p_n^\sigma(1) \geq H_\ell\, > \log(\ell).
\]
\end{lemma}

\begin{proof}
    This is immediate.
\end{proof}

Even though this estimate appears to be rather ``coarse'', it proves effective eventually. 

\begin{lemma}\label{lem:BoundP(4)}
Let \( n \geq 4 \), then
\begin{align*}
    p_n^\sigma(4) \geq \frac{1}{4!} \binom{n-1}{3} \,
\intertext{and in particular, if $n \geq 41$, we have}
    p_n^\sigma(4) > \frac{n^3}{4! \cdot 7}\, .
\end{align*}
    
\end{lemma}

\begin{proof}
    We have
    \[
        p_n^\sigma(4) = \frac{1}{4!} \sum_{i_1 + \cdots + i_4 = n} \sigma_{-1}(i_1) \cdots \sigma_{-1}(i_4) \geq \frac{1}{4!} \sum_{i_1 + \cdots + i_4 = n} 1 \,.
    \]
    A standard combinatorial argument shows that there are 
    \[
        \binom{n-4 + 3}{3} = \binom{n-1}{3} \,,
    \]
    compositions of \( n = i_1 + \cdots + i_4 \) into 4 parts with sizes \( i_1 , \ldots, i_4 \geq 1 \). So the first inequality follows.

    For the second one, we explicitly examine
    \[
    6 \cdot \left(7 \cdot \binom{n - 1}{3} - n^3\right) = 7 (n - 1)(n - 2)(n - 3) - 6n^3 = n^3 - 42n^2 + 77n - 42
    \]
    which is strictly positive for $n \geq 41$.
\end{proof}

\subsection{Upper bounds for \texorpdfstring{$p_n^\sigma(2)$}{p\_n\textasciicircum{}sigma(2)} and \texorpdfstring{$p_n^\sigma(3)$}{p\_n\textasciicircum{}sigma(3)}} In a recent preprint \cite{CharltonHeimStum}, the first, second, and fourth author provided the following bounds for $p_n^\sigma(2)$ and $p_n^\sigma(3)$.

\begin{prop}[\protect{\cite[Proposition 2.1]{CharltonHeimStum}}] \label{prop:BoundsFromSmallestCounterexample}
    Define
    \begin{align*}
        p_n^\sigma(2)_{\up} &\coloneqq \frac{1}{2!} \left(\frac{5}{2} (n-1) \sigma_{-3}(n) + (1 + \log(n))^2\right) \,;\\[1ex]
        p_n^\sigma(3)_{\up} &\coloneqq \begin{aligned}[t]\frac{1}{3!}\bigg( \frac{35(n-1)^2}{16} \sigma_{-5}(n) & {} + \frac{15(n-1)}{4} \sigma_{-4}(n) \zeta(3) (1 + \log(n)) \\[-1ex]
	    &{} + \frac{5}{2} (n-1) \sigma_{-3}(n) \zeta(2)  + n (1 + \log(n))^3\bigg)\, .\end{aligned}
    \end{align*}
    Then, for all $n \in \N$, we have
    \begin{align*}
        p_n^\sigma(2)_{\up} \geq p_n^\sigma(2) 
        \, ,\\*
        p_n^\sigma(3)_{\up} \geq p_n^\sigma(3) 
        \,.
    \end{align*}
\end{prop}

To exploit the upper bounds properly, we simplify the expressions via the following straightforward lemma whose proof we omit.

\begin{lemma}\label{lem:TechnicalBounds}
    For any $n \in \N$ and $k \geq 2$, we have
    \begin{equation*}
        \sigma_{-k}(n) < \zeta(k) < 1 + \frac{1}{k - 1}\, .
    \end{equation*}
    If $n \geq 303$, we also have
    \begin{equation*}
        \frac{1 + \log(n)}{\sqrt[3]{n}} < 1\, . 
    \end{equation*}
\end{lemma}

We now put these together.

\begin{lemma}\label{lem:BoundsP(2)P(3)}
    For $n \geq 303$, we get
    \begin{align*}
        p_n^\sigma(2) &\leq \frac{47}{24}n \,
    \intertext{as well as}
        p_n^\sigma(3) &\leq \frac{76\, 129}{116\, 352}n^2\, .
    \end{align*}
\end{lemma}

\begin{proof}
    Using the inequalities given in Lemma \ref{lem:TechnicalBounds} on the upper bound $p_n^\sigma(2)_\up$, we get
    \begin{align*}
        p_n^\sigma(2) &\leq \frac{1}{2!} \left(\frac{5}{2} (n-1) \sigma_{-3}(n) + (1 + \log(n))^2\right) \\
        &\leq \frac{1}{2} \left(\frac{5}{2} \cdot \frac{3}{2} + \frac{1}{6}\right)n = \frac{47}{24}n.
    \end{align*}
    where we used that $(1 + \log (n))^2 < \frac{n}{6}$. Similarly, we have
    \begin{align*}
        p_n^\sigma(3) &\leq \begin{aligned}[t]\frac{1}{3!}\bigg( \frac{35(n-1)^2}{16} \sigma_{-5}(n) & {} + \frac{15(n-1)}{4} \sigma_{-4}(n) \zeta(3) (1 + \log(n)) \\[-1ex]
	    &{} + \frac{5}{2} (n-1) \sigma_{-3}(n) \zeta(2)  + n (1 + \log(n))^3\bigg)\end{aligned}\\
        &\leq \frac{1}{3!} \left(\frac{35n^2}{16} \cdot \frac{5}{4} + \frac{15n}{4} \cdot \frac{4}{3} \cdot \frac{3}{2} \cdot \frac{n}{45} + \frac{5}{2} \cdot \frac{n^2}{303} \cdot \frac{3}{2} \cdot 2 + n^2\right) = \frac{76\, 129}{116\, 352}n^2
    \end{align*}
    where we used that $1 + \log(n) < \frac{n}{45}$.
\end{proof}

\section{Proof of the main result}\label{sec:proofhurwitz}

This section is devoted to the proof of Theorem \ref{thm:DArcaisNotHurwitz}.

\subsection{The reciprocal of the reduced D'Arcais polynomial} 
For $n\geq 2$, let
\[
f(X):=X^{n-1}R_n^\sigma\!\left(X^{-1}\right)
     =\sum_{k=1}^{n}p_n^\sigma(k)X^{n-k}.
\]
Note that for $z \in \C \setminus \{0\}$ with $\Re(z) > 0$, we also get $\Re(z^{-1}) > 0$. Hence, $f(X)$ is Hurwitz if and only if the same is true for $R_n^\sigma(X)$. Then
\[
\Delta_1(f)=p_n^\sigma(2)
\]
and
\[
\Delta_2(f)
=p_n^\sigma(2)p_n^\sigma(3) - p_n^\sigma(1)p_n^\sigma(4)\, .
\]

Consequently, if
\[
p_n^\sigma(1)p_n^\sigma(4) - p_n^\sigma(2)p_n^\sigma(3) >0,
\]
then $R_n^\sigma(X)$ has a zero in the open right complex half-plane.

\subsection{Proof of Theorem \ref{thm:DArcaisNotHurwitz}} We are now in a position to finalise our proof.

\begin{proof}[Proof of Theorem \ref{thm:DArcaisNotHurwitz}]
Using the bounds from above, we have that for \( n \geq 303 \)
\[
    p_n^\sigma(1)p_n^\sigma(4) - p_n^\sigma(2)p_n^\sigma(3)
    \geq \frac{\sigma_{-1}(n)n^3}{4! \cdot 7} - \frac{3\,578\,063n^3}{2\,792\,448} = \frac{116\, 352\sigma_{-1}(n) - 25\,046\,441}{19\,547\,136}\, .
\]
By Theorem \ref{thm:Hurwitz}, we deduce that $R_n^\sigma(X)$ is not Hurwitz if
\begin{equation}\label{eq:BoundForSigma-1ForNotHurwitz}
\sigma_{-1}(n) > \frac{25\,046\,441}{116\, 352} = 215.264\ldots \, .
\end{equation}
Since \( \sigma_{-1}(n) \) is unbounded by Lemma \ref{lem:BoundP(1)}, this inequality eventually holds, in particular by Lemma \ref{lem:BoundP(1)}, we may take $n = \ell!$ for an $\ell$ such that $H_\ell > \log(\ell) \geq \frac{25\,046\,441}{116\, 352}$. Hence, choose
\begin{align*}
\ell = \smash[b]{\left\lceil\exp\left(\frac{25\,046\,441}{116\, 352}\right)\right\rceil} = 
3\,&077\,030\,883\,517\,849\,244\,056\,113\,561\,601\,934\,895\,798\,281\, 720\,\\[-1ex]&971\,738\,852\,399\,001\,591\,620\,287\,219\,488\,043\,505\,919\,072\,573\,770\, .
\end{align*}
Taking the factorial of this number\footnote{Using \texttt{hypercalc} \url{https://www.mrob.com/pub/perl/hypercalc.html}} yields the claim for
\[
    \ell! = 10 ^ { 2.8632953142746\ldots \times 10 ^ {95}}\, .\footnote{This is a number with more digits than there are atoms in the observable universe \cite{EganLineweaver, PlanckCollab}.}
\]
As $\sigma_{-1}(n_1n_2) \geq \sigma_{-1}(n_1)$ for all $n_1, n_2 \in \N$, the inequality holds for all multiples of $\ell!$, too, and hence
\begin{equation*}
    \liminf_{n \to \infty} \frac{\mathcal{A}(n)}{n} \geq \frac{1}{\ell !} > 0
\end{equation*}
as claimed.
\end{proof}

\begin{rmks}\
    \begin{enumerate}
        \item As $\Delta_1(f) > 0$ by construction, our proof relied on $\Delta_2(f) < 0$ which is a sufficient criterion for $f$ not being Hurwitz. However, this is not necessary at all and hence, our approach does not allow any definite answer as to whether we found the smallest $n$ for which $R_n^\sigma(X)$ is not Hurwitz.
        \item For \( k \geq 2 \), the following asymptotic formula for the $k$-fold convolution  of \( \sigma_{-1} \) is given by Starr \cite{Starr}:
        \[
        (\overbrace{\sigma_{-1} \ast \cdots \ast \sigma_{-1}}^{k})(n) \sim \frac{n^{k-1} \zeta(2)^k \sigma_{-2k+1}(n)}{\Gamma(k)\zeta(2k)} \,, \text{as $ n \to \infty$}\,.
        \]
        In particular one is guided to the following improved upper and lower bounds, for any \( \epsilon > 0 \) and a sufficiently large \( n \):
        \begin{align*}
            p_n^\sigma(2) \leq \frac{1}{2!} \left( \frac{5}{2}  + \epsilon \right) \zeta(3)
            n
            \\
            p_n^\sigma(3) \leq \frac{1}{3!} \left( \frac{35}{16} + \epsilon \right) \zeta(5)
            n^2 
            \\
            p_n^\sigma(4) \geq \frac{1}{4!} \left( \frac{175}{144} - \epsilon \right) n^3
        \end{align*}
        The failure of the Hurwitz--Routh criterion is then witnessed when
        \begin{align*}
            \sigma_{-1}(n) &> \frac{\frac{1}{2!} \left( \frac{5}{2}  + \epsilon \right) \zeta(3) \cdot \frac{1}{3!} \left( \frac{35}{16} + \epsilon \right) \zeta(5)}{\frac{1}{4!} \left( \frac{175}{144} - \epsilon \right)} \\
            &= 9 \zeta(3) \zeta(5) + O(\epsilon)
        \end{align*}
        In this case, the \( 1\,563\textsuperscript{rd}\) superabundant number  
        \begin{align*}
        S_{1\,563} = {} 113\,853\,&868\,558\,315\,232\,571\,591\,404\,632\,806\,683\,080\,574\,547\,883\,715\,331\,\\
        &178\,785\,881\,099\,068\,163\,763\,133\,283\,612\,888\,011\,818\,919\,178\,613\,\\
        &625\,383\,130\,513\,557\,281\,140\,872\,437\,873\,408\,921\,476\,586\,393\,412\,\\
        &510\,334\,913\,628\,647\,234\,606\,393\,265\,427\,230\,012\,067\,709\,453\,218\,\\
        &606\,140\,328\,997\,865\,763\,463\,753\,286\,967\,791\,195\,387\,697\,920\,000 \,,
        \end{align*}
        has
        \[
            \sigma_{-1}(S_{1\,563}) = 11.21813\ldots > 9 \zeta(3)\zeta(5) = 11.21801\ldots \,,
        \]
        giving \( n = S_{1\,563} \) as a rough benchmark for the location of an explicit counterexample.
        \item Robin \cite{Robin} showed that the Riemann hypothesis is equivalent to
        \[
        \sigma_{-1}(n) < e^\gamma \log(\log(n))
        \]
        for all $n > 5040$ where $\gamma$ is the Euler--Mascheroni constant. Hence under the Riemann hypothesis, the inequality in Eqn. \eqref{eq:BoundForSigma-1ForNotHurwitz} is false for all $n$ less than $10^{10^{52}}$ which exceeds the benchmark above by several orders of magnitude and underlines the sheer size of our error terms.
    \end{enumerate}
\end{rmks}

\section{Open challenges}

Figure \ref{fig:ComplexRoots} suggests certain regularities in the distribution of the non-real zeros of $R_n^\sigma(X)$. At present, however, it is unclear whether the apparent curves reflect a genuine asymptotic phenomenon. In particular, the numerical data shown in the figure do not indicate how the zeros in the open right half-plane whose existence follows from Theorem \ref{thm:DArcaisNotHurwitz} are situated. Theorem \ref{thm:DArcaisNotHurwitz} is qualitative in this respect: it neither locates these zeros nor determines their number. It would therefore be interesting to characterise the indices n for which $R_n^\sigma(X)$ has a zero with positive real part, to obtain quantitative information on the number and location of such zeros, and to decide whether the zeros occurring in the open right half-plane comprise infinitely many pairwise distinct complex numbers.

We are also interested in the cases for which the reduced D’Arcais polynomial is Hurwitz. In particular, is $R_p^\sigma(X)$ Hurwitz for every prime $p$?

Recently, the second and third author \cite{HeimNeuDominant} showed that the root of $P_n^\sigma(X)$ with largest absolute value is real and simple. Numerical evidence suggests that all zeros are simple but this remains an open question.

\bibliographystyle{habbrv2}
\bibliography{bibliography.bib}

@article{Routh,
    author = {Edward John Routh},
    title = {A Treatise on the Stability of a Given State of Motion},
    journal = {Macmillan},
    year = {1877}
}

@misc{HeimNeuDominant,
    author = {Heim, B. and Neuhauser, M.},
    title = {Dominant Zeros of {N}ekrasov--{O}kounkov Polynomials},
    year = {2026},
    eprint = {2606.15394},
    archivePrefix = {arXiv},
    comment = {with Appendix by K. Ono}
}

@article{Robin,
    author =  {Robin, G.},
    title = {Grandes valeurs de la fonction somme des diviseurs et hypothèse de Riemann},
    journal = {Journal de Mathématiques Pures et Appliquées},
    year = {1984},
    volume =  {63},
    pages = {187--213}
}

@article{EganLineweaver,
    author = {Egan, C. A. and Lineweaver, C. H.},
    title = {A LARGER ESTIMATE OF THE ENTROPY OF THE UNIVERSE},
    journal = {The Astrophysical Journal},
    volume = {710},
    number = {2},
    year = {2010},
    pages = {1825--1834},
    doi = {10.1088/0004-637X/710/2/1825}
    }

@article{PlanckCollab,
    author = {{The Planck Collaboration}},
    title = {Planck 2018 results VI},
    subtitle = {Cosmological parameters},
    journal = {Astronomy and Astrophysics},
    year = {2020},
    volume = {641},
    number = {A6},
    doi = {10.1051/0004-6361/201833910}
}

@article{Hurwitz,
    author = {Hurwitz, A.},
    title = {Ueber die {B}edingungen, unter welchen eine {G}leichung nur {W}urzeln mit negativen reellen {T}heilen besitzt},
    journal = {Mathematische Annalen},
    year = {1895},
    doi = {10.1007/BF01446812},
    issue = {46},
    pages = {273--284}
}

@misc{stumpenhusen2026logconcavitydarcaispolynomialsnormalised,
      title={{O}n the {L}og-{C}oncavity of the {D}'{A}rcais {P}olynomials for {N}ormalised {F}unctions}, 
      author={Johann Stumpenhusen},
      year={2026},
      eprint={2607.14961},
      archivePrefix={arXiv},
      primaryClass={math.NT}
}

@article{Serre,
    author = {Serre, J.-P.},
    title = {Sur la lacunarité des puissances de $\eta$},
    journal = {Glasgow Mathematical Journal},
    year = {1985},
    issue = {27},
    pages = {203-221},
    doi = {10.1017/S0017089500006194}
}

@misc{CharltonHeimStum,
      title={{O}n the {S}mallest {C}ounterexample to the {L}og-{C}oncavity of the {D}'{A}rcais {P}olynomials}, 
      author={Steven Charlton and Bernhard Heim and Johann Stumpenhusen},
      year={2026},
      eprint={2606.09545},
      archivePrefix={arXiv},
      primaryClass={math.NT}
}

@misc{Starr,
    author =  {Shannon Starr},
    title = {Asymptotics of the {D}'{A}rcais Numbers at Small $k$},
    year = {2026},
    eprinttype = {arXiv},
    eprint = {2601.18599}
}

@article {HeimNeu20,
    AUTHOR = {Heim, Bernhard and Neuhauser, Markus},
     TITLE = {Horizontal and vertical log-concavity},
   JOURNAL = {Res. Number Theory},
  FJOURNAL = {Research in Number Theory},
    VOLUME = {7},
      YEAR = {2021},
    NUMBER = {1},
     PAGES = {Paper No. 18, 12},
      ISSN = {2522-0160,2363-9555},
   MRCLASS = {11B83 (05A10 11B37)},
  MRNUMBER = {4220048},
MRREVIEWER = {Eric\ S.\ Egge},
       DOI = {10.1007/s40993-021-00245-1},
       URL = {https://doi.org/10.1007/s40993-021-00245-1},
       eprinttype = {arxiv},
       eprint = {2010.05231}
}

@article {HeimStum26,
    AUTHOR = {Heim, Bernhard and Stumpenhusen, Johann},
     TITLE = {On the detection of non-roots of {D}'{A}rcais polynomials},
   JOURNAL = {Res. Number Theory},
  FJOURNAL = {Research in Number Theory},
    VOLUME = {12},
      YEAR = {2026},
    NUMBER = {2},
     PAGES = {Paper No. 42, 11},
      ISSN = {2522-0160,2363-9555},
   MRCLASS = {05A17 (05A20 11R04 11R09)},
  MRNUMBER = {5058125},
       DOI = {10.1007/s40993-026-00705-6},
       URL = {https://doi.org/10.1007/s40993-026-00705-6},
       eprinttype = {arXiv},
       eprint = {2511.16276}
}

@article {HeimNeu25,
    AUTHOR = {Heim, Bernhard and Neuhauser, Markus},
     TITLE = {On the non-vanishing of the {D}'{A}rcais polynomials},
   JOURNAL = {Ramanujan J.},
  FJOURNAL = {Ramanujan Journal. An International Journal Devoted to the
              Areas of Mathematics Influenced by Ramanujan},
    VOLUME = {69},
      YEAR = {2026},
    NUMBER = {1},
     PAGES = {Paper No. 7, 11},
      ISSN = {1382-4090,1572-9303},
   MRCLASS = {11F30 (05A17 05A20 11P82 11R04)},
  MRNUMBER = {5004687},
       DOI = {10.1007/s11139-025-01286-1},
       URL = {https://doi.org/10.1007/s11139-025-01286-1},
       eprinttype = {arXiv},
       eprint = {2509.06123}
}

@article {AbPruDoVe,
    AUTHOR = {Abdesselam, Abdelmalek and Brunialti, Pedro and Doan, Tristan
              and Velie, Philip},
     TITLE = {A bijection for tuples of commuting permutations and a
              log-concavity conjecture},
   JOURNAL = {Res. Number Theory},
  FJOURNAL = {Research in Number Theory},
    VOLUME = {10},
      YEAR = {2024},
    NUMBER = {2},
     PAGES = {Paper No. 45, 10},
      ISSN = {2522-0160,2363-9555},
   MRCLASS = {05A20 (05A05 20B30)},
  MRNUMBER = {4734058},
MRREVIEWER = {Niklas\ Eriksen},
       DOI = {10.1007/s40993-024-00531-8},
       URL = {https://doi.org/10.1007/s40993-024-00531-8},
       eprinttype = {arXiv},
       eprint = {2309.09407}
}

@article {Abdessel23,
    AUTHOR = {Abdesselam, Abdelmalek},
     TITLE = {Log-concavity with respect to the number of orbits for
              infinite tuples of commuting permutations},
   JOURNAL = {Ann. Comb.},
  FJOURNAL = {Annals of Combinatorics},
    VOLUME = {29},
      YEAR = {2025},
    NUMBER = {2},
     PAGES = {563--573},
      ISSN = {0218-0006,0219-3094},
   MRCLASS = {05A20 (05A05)},
  MRNUMBER = {4913024},
MRREVIEWER = {Niklas\ Eriksen},
       DOI = {10.1007/s00026-024-00724-z},
       URL = {https://doi.org/10.1007/s00026-024-00724-z},
       eprinttype = {arXiv},
       eprint = {2309.07358}
}

@article {Leh47,
    AUTHOR = {Lehmer, D. H.},
     TITLE = {The vanishing of {R}amanujan's function {$\tau(n)$}},
   JOURNAL = {Duke Math. J.},
  FJOURNAL = {Duke Mathematical Journal},
    VOLUME = {14},
      YEAR = {1947},
     PAGES = {429--433},
      ISSN = {0012-7094,1547-7398},
   MRCLASS = {10.0X},
  MRNUMBER = {21027},
MRREVIEWER = {R.\ A.\ Rankin},
       URL = {http://projecteuclid.org/euclid.dmj/1077474140},
       DOI = {10.1215/S0012-7094-47-01436-1}
}

@article{DAr13,
  title={D{\'e}veloppement en s{\'e}rie},
  author={D’Arcais, F},
  journal={Interm{\'e}diaire Math},
  volume={20},
  pages={233--234},
  year={1913}
}

@incollection {NO06,
    AUTHOR = {Nekrasov, Nikita A. and Okounkov, Andrei},
     TITLE = {Seiberg-{W}itten theory and random partitions},
 BOOKTITLE = {The unity of mathematics},
    SERIES = {Progr. Math.},
    VOLUME = {244},
     PAGES = {525--596},
 PUBLISHER = {Birkh\"auser Boston, Boston, MA},
      YEAR = {2006},
      ISBN = {978-0-8176-4076-7; 0-8176-4076-2},
   MRCLASS = {81T60 (05E10 11Z05 14D21 60C05 81T45)},
  MRNUMBER = {2181816},
MRREVIEWER = {Johan\ A.\ Martens},
       DOI = {10.1007/0-8176-4467-9\_15},
       URL = {https://doi.org/10.1007/0-8176-4467-9_15},
       eprinttype = {arXiv},
       eprint = {hep-th/0306238}
}

@article {Zh22,
    AUTHOR = {Zhang, Shengtong},
     TITLE = {Log-concavity in powers of infinite series close to
              {$(1-z)^{-1}$}},
   JOURNAL = {Res. Number Theory},
  FJOURNAL = {Research in Number Theory},
    VOLUME = {8},
      YEAR = {2022},
    NUMBER = {4},
     PAGES = {Paper No. 66, 17},
      ISSN = {2522-0160,2363-9555},
   MRCLASS = {05A15 (11P82)},
  MRNUMBER = {4483568},
MRREVIEWER = {Eric\ S.\ Egge},
       DOI = {10.1007/s40993-022-00370-5},
       URL = {https://doi.org/10.1007/s40993-022-00370-5},
       eprinttype = {arXiv},
       eprint = {2203.12008}
}

@article {Ha10,
    AUTHOR = {Han, Guo-Niu},
     TITLE = {The {N}ekrasov-{O}kounkov hook length formula: refinement,
              elementary proof, extension and applications},
   JOURNAL = {Ann. Inst. Fourier (Grenoble)},
  FJOURNAL = {Universit\'e{} de Grenoble. Annales de l'Institut Fourier},
    VOLUME = {60},
      YEAR = {2010},
    NUMBER = {1},
     PAGES = {1--29},
      ISSN = {0373-0956,1777-5310},
   MRCLASS = {05A17 (05A15 11P82)},
  MRNUMBER = {2664308},
MRREVIEWER = {Christine\ Bessenrodt},
       DOI = {10.5802/aif.2515},
       URL = {https://doi.org/10.5802/aif.2515},
       eprinttype = {arXiv},
       eprint = {0805.1398}
}

@article {HeimNeu19Hooklength,
    AUTHOR = {Heim, Bernhard and Neuhauser, Markus},
     TITLE = {On conjectures regarding the {N}ekrasov-{O}kounkov hook length
              formula},
   JOURNAL = {Arch. Math. (Basel)},
  FJOURNAL = {Archiv der Mathematik},
    VOLUME = {113},
      YEAR = {2019},
    NUMBER = {4},
     PAGES = {355--366},
      ISSN = {0003-889X,1420-8938},
   MRCLASS = {05A17 (05A19 11P82)},
  MRNUMBER = {4008070},
MRREVIEWER = {Mircea\ Merca},
       DOI = {10.1007/s00013-019-01335-4},
       URL = {https://doi.org/10.1007/s00013-019-01335-4},
       eprinttype = {arXiv},
       eprint = {1810.02226},
}

@article {Adm2018,
    AUTHOR = {Adm, Mohammad and Garloff, J\"urgen and Tyaglov, Mikhail},
     TITLE = {Total nonnegativity of finite {H}urwitz matrices and root
              location of polynomials},
   JOURNAL = {J. Math. Anal. Appl.},
  FJOURNAL = {Journal of Mathematical Analysis and Applications},
    VOLUME = {467},
      YEAR = {2018},
    NUMBER = {1},
     PAGES = {148--170},
      ISSN = {0022-247X,1096-0813},
   MRCLASS = {15B48},
  MRNUMBER = {3834798},
MRREVIEWER = {Dominique\ Guillot},
       DOI = {10.1016/j.jmaa.2018.06.065},
       eprint = {1711.04651},
       eprinttype = {arXiv}
}

\vspace{-1.5em}

\end{document}